\documentclass{amsart}
\usepackage{graphicx}
\usepackage{amsmath}
\usepackage{amssymb}%
\usepackage{amsfonts}%
\usepackage{graphicx}
\usepackage[all]{xy}
\usepackage[mathscr]{euscript}
\usepackage{hyperref}
\newtheorem{theorem}{Theorem}[section]
\newtheorem{lemma}[theorem]{Lemma}

\theoremstyle{definition}
\newtheorem{definition}[theorem]{Definition}

\newtheorem{proposition}[theorem]{Proposition}
\newtheorem{corollary}[theorem]{Corollary}

\usepackage[all]{xy}
\usepackage{graphicx}
\theoremstyle{remark}

\numberwithin{equation}{section}
\usepackage{graphicx}

\begin{document}

\title{ \rm A Geometric Approach to Shortest Bounded Curvature Paths}

\author{Jos\'{e} Ayala}
\address{FIA, Universidad Arturo Prat, Iquique, Chile}
\email{jayalhoff@gmail.com}

\author{David Kirszenblat}
\address{Department of Mathematics and Statistics, University of Melbourne
              Parkville, VIC 3010 Australia}
\email{d.kirszenblat@student.unimelb.edu.au}

\author{Hyam Rubinstein}
\address{Department of Mathematics and Statistics, University of Melbourne
              Parkville, VIC 3010 Australia}
\email{rubin@ms.unimelb.edu.au}


\maketitle

\begin{abstract} We present a geometric proof for the classification of length minimisers in spaces of planar bounded curvature paths. Our methods can be adapted without much effort to classify length minimisers in spaces of bounded curvature paths in other surfaces. The main result in this note fills a gap by furnishing a geometric proof for a problem in geometry that hither to was missing from the literature.
\end{abstract}

\section{Motivation}
In this work we present a constructive proof for the classification of global length minimisers in spaces of planar bounded curvature paths first obtained by L. Dubins in \cite{dubins 1}. Our approach is flexible and easy to generalise. In particular, the techniques developed in this note allowed us to achieve:

\begin{itemize}  
\item the classification of the homotopy classes of bounded curvature paths \cite{paperd};
\item the classification of length minimisers in homotopy classes of bounded curvature paths \cite{paperb}.
\end{itemize}

In general terms, a planar bounded curvature path corresponds to a $C^1$ and piecewise $C^2$ path lying in $\mathbb R^2$. These paths connect two elements in the tangent bundle $T \mathbb R^2$ and have curvature bounded by a positive constant $\kappa$.  As proved by L. Dubins in \cite{dubins 1} a $C^1$  path of minimal length lying in ${\mathbb R}^2$, having its curvature bounded by a positive constant, connecting two elements of the tangent bundle $T{\mathbb R}^2$ is indeed a piecewise $C^2$ path (this is the well known {\sc csc}, {\sc ccc} characterisation \cite{dubins 1}). The view put forward by L. Dubins regarding the bound on the curvature (an interpretation followed by many) comes from considering a particle that moves at constant speed subject to a maximum possible force. Another physical interpretation can be the following: a uniform bending beam has an internal structure which depends on certain binding forces; when exposed to tension the beam may be deformed - such deformations always have a limit given by the beam's material resistance. Length minimising bounded curvature paths, widely known as Dubins paths, have proven to be extraordinarily useful in applications since a bound on the curvature is a turning circle constraint for the trajectory of a vehicle along a path. In contrast, our approach deals with curves as topological objects. 

To characterise the length minimisers we proceed as follows. Start with an arbitrary bounded curvature path and consider a partition of it so that the length of each piece is less than $\frac{1}{\kappa}$. Replace each piece of the path, called a {\it fragment}, by a piecewise constant curvature path, called a {\it replacement}, making sure the length of each replacement is at most the length of the respective fragment to be replaced. As a consequence of the previously described process, we obtain a bounded curvature path corresponding to a finite number of concatenated piecewise constant curvature paths (a $cs$ path) called a {\it normalisation}. Notice that there is an implicit sense of continuity when replacing a bounded curvature path by a normalisation since fragments can be chosen to be arbitrarily small. However, a continuity argument preserving the bound on curvature is only achieved in \cite{paperd}. After obtaining a normalisation for a given path, we develop a recursive process to decrease the path {\it complexity} (number of constant curvature components in a $cs$ path), and the path length, to then conclude that the length minimisers have complexity at most 3. There is also a homotopy involved in this process but no claim about curvature is needed.

In 1887 A. Markov in \cite{markov} initiated this theory by studying the optimal linking of railroad tracks. Seventy years later L. Dubins in \cite{dubins 1, dubins 2} obtained the first general results in this theory. In between the work of L. Dubins and A. Markov there are a number of scattered results due to V. Ionin, G. Pestov, E. Schmidt, A. Schur and H. Schwarz \cite{blas, pestov, schm}.  There are a number of proofs for the characterisation of length minimisers, each of them with a different flavour. In 1974 J. Johnson in \cite{johnson} characterised length minimisers by applying the Pontryagin maximum principle. Later in 1990 Reeds and Shepp studied the length minimisers for the trajectories of a car moving forward and backward. Recently S. Eriksson-Bique et. al. in \cite{kirk} introduced a discrete analogue to Dubins paths. In \cite{mon} F. Monroy-P\'{e}rez also applied the  Pontryagin maximum principle to obtain the length minimisers in 2-dimensional homogeneous spaces of constant curvature. Our proof is elementary, easy to implement and provides a foundation for a continuity argument for the deformation of bounded curvature paths \cite{paperd}. This step permits us to push the theory forward outside optimality. In addition, our techniques establish an elementary framework to study bounded curvature paths in other surfaces. 

Bounded curvature paths have proven to be useful in science and engineering as well as in mathematics. These paths have been extensively studied in computer science \cite{aga, bui, kirk, ny, rus}, control theory \cite{jur, mit, mon, mur, sig1, sig2, sus2}, engineering \cite{brazil 1, chang, irina1, san} and recently in geometric knot theory and the Steiner tree problem \cite{sul1, dur1,elk, dave1}.  It would be interesting to study the computational complexity of reducing a $cs$ path to the length minimiser given a set of operations on $cs$ paths reducing the length and complexity. Currently the first author is implementing the techniques presented in this work in Dubins Explorer, a software for bounded curvature paths \cite{dubinsexplorer}. In general, a bound on curvature is a property widely observed in nature: in potamology, in the formation of meanders \cite{meanders1} and in the geometry of coral structures \cite{coral}.

\section{Normalising Bounded Curvature Paths}

We start this work by presenting basic definitions. Then, we procede to introduce the concept of normalisation for bounded curvature paths. The idea is to (after some point) only deal with bounded curvature paths corresponding to a finite number of concatenations of line segments and arcs of unit circles. Let us denote by $T{\mathbb R}^2$ the tangent bundle of ${\mathbb R}^2$. The elements in $T{\mathbb R}^2$ correspond to pairs $(x,X)$ sometimes denoted just by {\sc x}. As usual, the first coordinate corresponds to a point in ${\mathbb R}^2$ and the second to a tangent vector in ${\mathbb R}^2$ at $x$.

\begin{definition} \label{adm_pat} Given $(x,X),(y,Y) \in T{\mathbb R}^2$, we say that a path $\gamma: [0,s]\rightarrow {\mathbb R}^2$ connecting these points is a {\it bounded curvature path} if:
\end{definition}
 \begin{itemize}
\item $\gamma$ is $C^1$ and piecewise $C^2$.
\item $\gamma$ is parametrized by arc length (i.e. $||\gamma'(t)||=1$ for all $t\in [0,s]$).
\item $\gamma(0)=x$,  $\gamma'(0)=X$;  $\gamma(s)=y$,  $\gamma'(s)=Y.$
\item $||\gamma''(t)||\leq \kappa$, for all $t\in [0,s]$ when defined, $\kappa>0$ a constant.
\end{itemize}
Of course, $s$ is the arc-length of $\gamma$.

The first condition means that a bounded curvature path has continuous first derivative and piecewise continuous second derivative. For the third condition to make sense, without loss of generality, we extend the domain of $\gamma$ to $(-\epsilon,s+\epsilon)$ for $\epsilon>0$. The third item is called the endpoint condition. The bound on the curvature $\kappa>0$ can be chosen to be $\kappa=1$ by considering a suitable scaling of the plane. Generally, the interval $[0,s]$ is denoted by $I$. Denote by ${\mathcal L}(\gamma)$ the length of $\gamma$ and ${\mathcal L}(\gamma,a,b)$ the length of $\gamma$ restricted to $[a,b]\subset I$.

\begin{definition} \label{admsp} Given $\mbox{\sc x,y}\in T{\mathbb R}^2$. The space of bounded curvature paths starting at $x$ tangent to $X$ finishing at $y$ tangent to $Y$ is denoted by $\Gamma(\mbox{\sc x,y})$. \end{definition}

Important properties of $\Gamma(\mbox{\sc x,y})$ (e.g., the kind of length minimiser, the number of length minimisers, the number of connected components) depend on the endpoint condition. Also, the existence of a homotopy class in $\Gamma(\mbox{\sc x,y})$ whose elements are embedded, depends on the chosen endpoint \cite{paperc, paperd}. 

Consider the origin of an orthogonal coordinate system in $\mathbb R^2$ as the base point $x$ with $X$ lying in the abscissa. 
Note that the arc-length parametrisation implies the endpoints are elements in the unit tangent bundle $UT{\mathbb R}^2$. This space is equipped with a natural projection $\pi : UT{\mathbb R}^2 \rightarrow {\mathbb R}^2$. The fiber $\pi^{-1}(x)$ corresponds to ${\mathbb S}^1$ for all $x \in{\mathbb R}^2$. The space of endpoint conditions corresponds to a sphere bundle on ${\mathbb R}^2$ with fiber ${\mathbb S}^1$.

 Observe that in a bounded curvature path the osculating circle at each point (when defined) has radius at least $1$. In Figure \ref{BCPex2} we illustrate some bounded curvature paths and osculating circles.

\begin{definition} Let $\mbox{\sc C}_ l(\mbox{\sc x})$ be the unit circle tangent to $x$ and to the left of $X$. An analogous interpretation applies for $\mbox{\sc C}_ r(\mbox{\sc x})$, $\mbox{\sc C}_ l(\mbox{\sc y})$ and $\mbox{\sc C}_ r(\mbox{\sc y})$ (see Figure \ref{BCPex2}). These circles are called {\it adjacent circles}. We denote their centers with lower-case letters, so the center of $\mbox{\sc C}_ l(\mbox{\sc x})$ is denoted by $c_l(\mbox{\sc x})$.
  \end{definition}

{ \begin{figure} [[htbp]
 \begin{center}
\includegraphics[width=1\textwidth,angle=0]{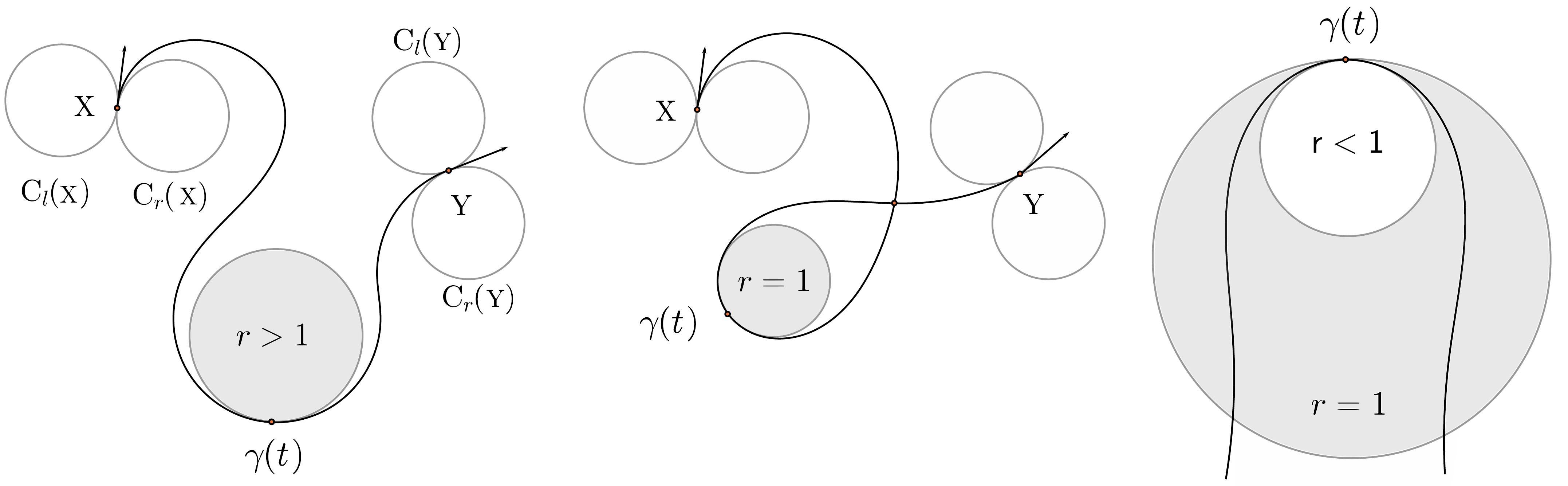}
\end{center}
\caption{Left and centre: Examples of bounded curvature paths and adjacent circles. Right: An example when the curvature bound is violated. Note that the radius of the osculating circle $r<1$ at   $\gamma(t)$ implies $\kappa>1$. }
 \label{BCPex2}
\end{figure}}

\begin{definition} A $cs$ path is a bounded curvature path corresponding to a finite number of concatenations of line segments (denoted by {\sc s}) and arcs of a unit circle (denoted by {\sc c}) see Figure \ref{BCPex2}. Let {\sc r} denote a clockwise traversed arc and {\sc l} a counterclockwise traversed arc. The line segments and arcs of circles are called  {\it components}. The number of components is called the {\it complexity} of the path.
\end{definition}

{ \begin{figure} [[htbp]
 \begin{center}
\includegraphics[width=.8\textwidth,angle=0]{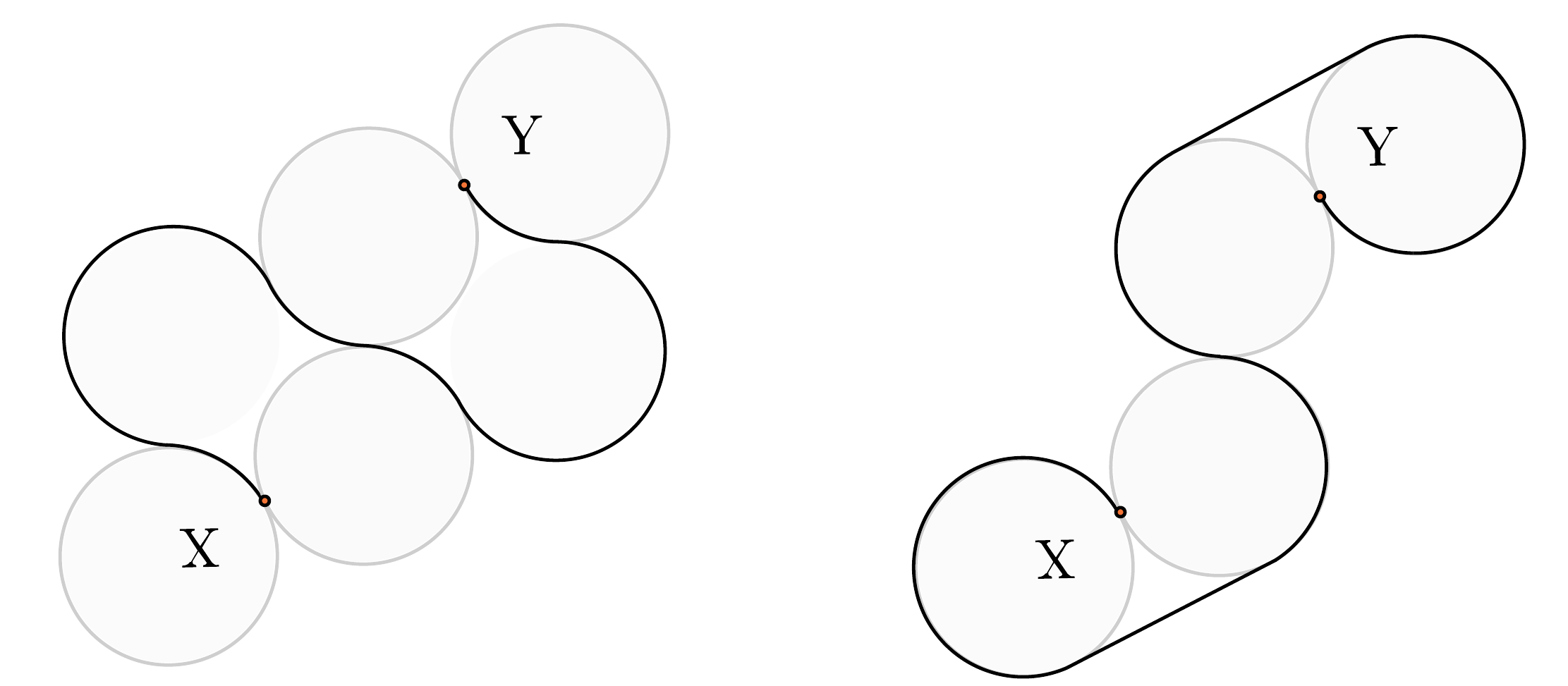}
\end{center}
\caption{Examples of $cs$ paths in $\Gamma(\mbox{\sc x,y})$.}
 \label{figcomplexity}
\end{figure}}

\begin{definition}\label{frag} A {\it fragmentation} of a bounded curvature path $\gamma:I \rightarrow {\mathbb R}^2$ corresponds to a finite sequence $0=t_0<t_1\ldots <t_m=s$ of elements in $I$ such that,
${\mathcal L}(\gamma,t_{i-1},t_i)<  1 $
with
$\sum_{i=1}^m {\mathcal L}(\gamma,t_{i-1},t_i) =s$.
We denote by a {\it fragment}, the restriction of $\gamma$ to the interval determined by two consecutive elements in the fragmentation.
\end{definition}

\begin{definition}\label{rzlip} Let ${\mbox{\sc z}}\in T{\mathbb R}^2$ with $\mbox{\sc z}=(z,Z)$, where the first component corresponds to the origin of a coordinate system with abscissa $x$ and ordinate $y$, and $Z=(1,0)$. Denote by ${\mathcal R}(\mbox{\sc z})$ the region enclosed by the complement of the union of the interior of the disks with boundary $\mbox{\sc C}_l(\mbox{\sc z})$ and $\mbox{\sc C}_r(\mbox{\sc z})$ intersected with a unit radius disk centered at $z$ denoted by ${D}_z$. Let ${\mathcal R}^+(\mbox{\sc z})$ be the set of points in ${\mathcal R}(\mbox{\sc z})$ with $x>0$ and ${\mathcal R}^-(\mbox{\sc z})$ be the set of points in ${\mathcal R}(\mbox{\sc z})$ with $x<0$ (see Figure \ref{figleavreg}).
\end{definition}

The next proposition uses a technical result. We propose that the reader refer to Corollary 2.4 in \cite{papere} for details. The omission of such a result should not compromise the reader's understanding. Corollary 2.4 in \cite{papere} states that a bounded curvature path making a u-turn must have length at least $2$ (see dashed trace in Figure \ref{figleavreg}). 

\begin{proposition} \label{r1r3p} Given $\mbox{\sc z}\in T{\mathbb R}^2$. A fragment $\gamma: I\to {\mathcal R}(\mbox{\sc z})$ not entirely lying in $\partial {\mathcal R}(\mbox{\sc z})$ with $\gamma(t)=z$ and $\gamma'(t)=Z$ does not intersect $\partial{\mathcal R}^-(\mbox{\sc z})$ or $\partial{\mathcal R}^+(\mbox{\sc z})$ (see Figure \ref{figleavreg} right).
\end{proposition}

\begin{proof} Choose $\mbox{\sc z}\in T{\mathbb R}^2$ and suppose that a fragment $\gamma$ is contained in ${\mathcal R}(\mbox{\sc z})$ with $\gamma(t)=z$ and $\gamma'(t)=Z$. Without loss of generality suppose $\gamma$ intersects $\partial^+{\mathcal R}(\mbox{\sc z})$ in $\mbox{\sc C}_r(\mbox{\sc z})$ at $\gamma(t')$ for some $t'\in I$. By considering $z=p$ and $\gamma(t')=q$ as in Corollary 2.4 in \cite{papere} we obtain that the length of $\gamma$ is at least 2. Since the diameter of $\partial{\mathcal R}^+(\mbox{\sc z})$ is equal to 1 we conclude the proof.  The other cases are proved by applying a similar argument.
\end{proof}

{ \begin{figure} [[htbp]
 \begin{center}
\includegraphics[width=.8\textwidth,angle=0]{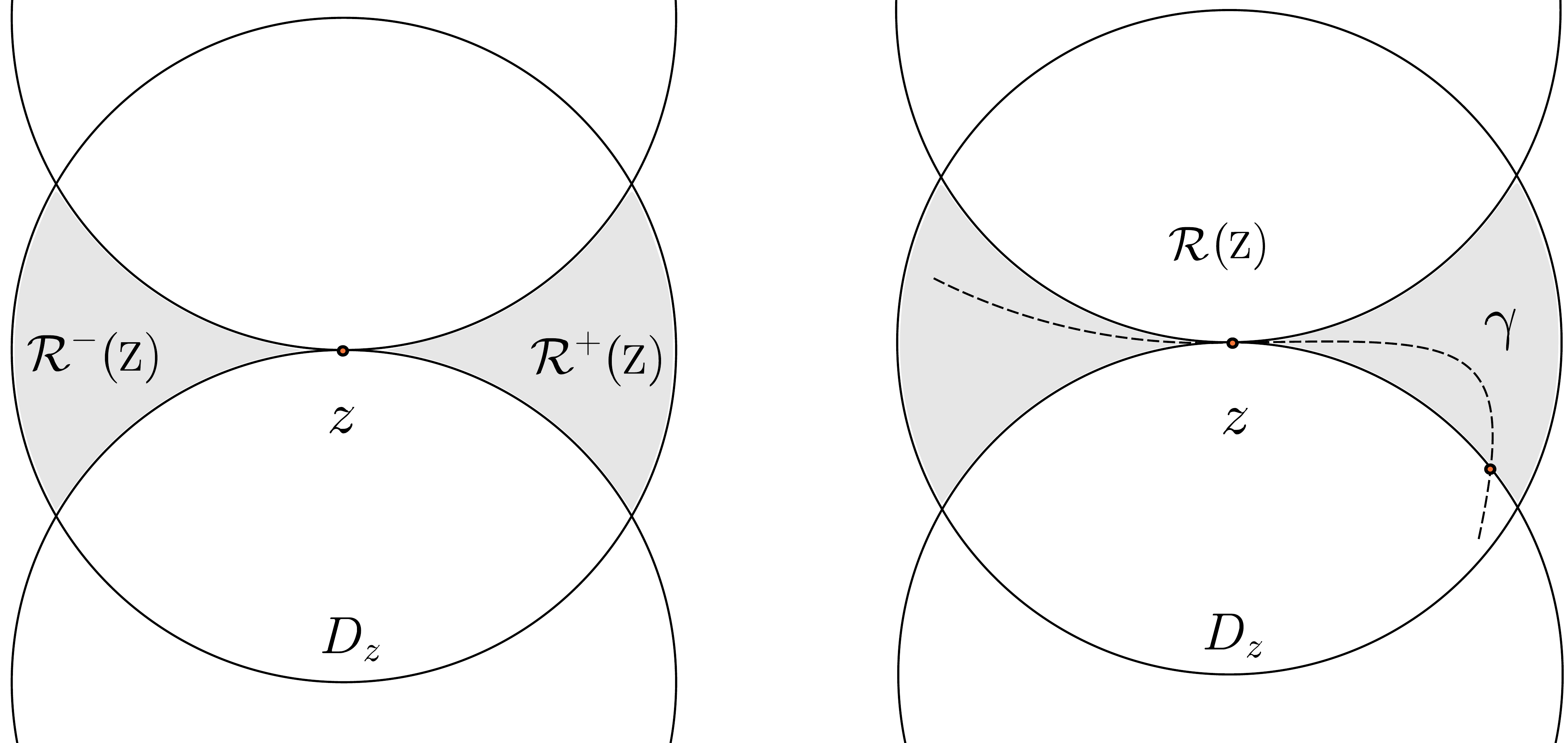}
\end{center}
\caption{Left: The grey region corresponds to ${\mathcal R}(\mbox{\sc z})$. Right: The dashed trace is not a fragment since by Proposition \ref{r1r3p} a fragment never leaves ${\mathcal R}(\mbox{\sc z})$ through $\mbox{\sc C}_l(\mbox{\sc z})$ or $\mbox{\sc C}_r(\mbox{\sc z})$. }
\label{figleavreg}
\end{figure}}

We include the proofs for the following simple but crucial lemmata. These give lower bounds for the lengths of curves when compared with arcs of unit circles and line segments. Consider a curve $\gamma (t)=(r(t)\cos \theta(t), r(t)  \sin \theta(t))$ in polar coordinates.

\begin{lemma}\label{rad} \rm For any curve $\gamma :[0,s] \rightarrow {\mathbb R}^2$ with $\gamma(0)= (1,0)$, $r(t)\geq 1$, and $\theta(s)=\eta$, one has ${\mathcal L}(\gamma)\geq \eta$ (see Figure \ref{figlemmarc} left).
\end{lemma}
\begin{proof} Consider $\gamma$ as defined above with $r,\theta:[0,s] \rightarrow {\mathbb R}$. Then we immediately have that:
$${\mathcal L}(\gamma)=\int_0^{s}\,|\gamma'(t)|\,dt =\int_0^{s} \sqrt{{|r'(t)|^2+ r(t)^2}\,|\theta'(t)|^2}\,\, dt \geq  \int_0^{s} { |{\theta}'(t)|}\,\, dt\geq{\eta}$$
\end{proof}

\begin{lemma}\label{seg}\rm  For any $C^1$ curve $\gamma :[0,s]\rightarrow {\mathbb R}^2$ with $\gamma(0)=(0,0)$, $\gamma(s)=(x,z)$ and $z\geq 0$, one has ${\mathcal L}(\gamma) \geq z$ (see Figure \ref{figlemmarc} right).

\end{lemma}
\begin{proof} Consider $\gamma (t)=(x(t),y(t))$ with $x,y:[0,s] \rightarrow {\mathbb R}$. Then,
$${\mathcal L}(\gamma)=\int_0^s\,|\gamma'(t)|\,dt =\int_0^s \sqrt{x'(t)^2+ y'(t)^2}\,\, dt \geq \int_0^s |{y'(t)}|\,\, dt \geq z$$
\end{proof}

\begin{definition} A bounded curvature path $\gamma: I \rightarrow {\mathbb R}^2$ in $\Gamma(\mbox{\sc x,y})$ has a negative direction if there exists $t\in I$ such that  $\langle X, \gamma'(t)\rangle<0$.
 \end{definition}

\begin{corollary}\label{nonedir} A fragment does not have a negative direction.
\end{corollary}
\begin{proof} It is easy to see from Lemma \ref{rad} that a negative direction implies the length of a bounded curvature path is greater than $\frac{\pi}{2}$. Since fragments have length less than $1$ the result follows.
\end{proof}

{ \begin{figure} [[htbp]
 \begin{center}
\includegraphics[width=.7\textwidth,angle=0]{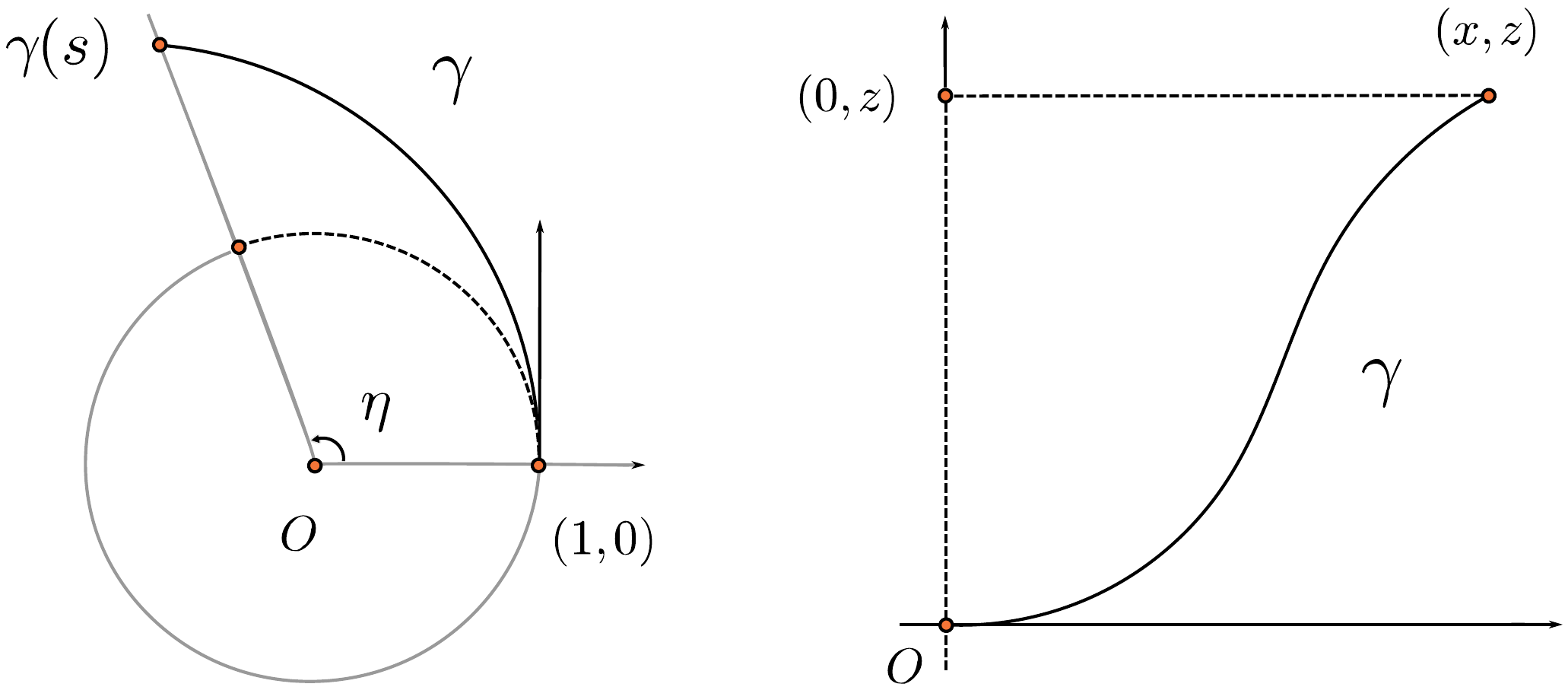}
\end{center}
\caption{Left: Illustration of Lemma \ref{rad}. Right: Illustration of Lemma \ref{seg}.}
 \label{figlemmarc}
\end{figure}}

In \cite{paperc} we proved that for certain endpoints in $T{\mathbb R}^2$ there exists a compact region $\Omega\subset{\mathbb R}^2$ that {\it traps} embedded bounded curvature paths (see Figure \ref{figregform} right). That is, no embedded bounded curvature lying in $\Omega$ can be deformed while preserving the curvature bound to a path having a point in the complement of $\Omega$.

\begin{proposition} \label{r1r3pr} A fragment $\gamma$ such that $\gamma(t)=z$ and $\gamma'(t)=Z$ is contained in ${\mathcal R}(\mbox{\sc z})$. \end{proposition}

\begin{proof} Suppose a fragment $\gamma$ intersects the boundary of the unit disk $D_z$ at $P=\gamma(t)$ for some $t\in I$. By virtue of Lemma \ref{seg} the length of $\gamma$ is greater than or equal to the length of $\overline{zP}=1$ leading to a contradiction (see Figure \ref{figregform} left). By Proposition \ref{r1r3p} fragments inside ${\mathcal R}(\mbox{\sc z})$ that are not arcs of an adjacent circle do not intersect $\partial{\mathcal R}(\mbox{\sc z})$. We conclude that fragments are confined in ${\mathcal R}(\mbox{\sc z})$.
 \end{proof}

{ \begin{figure} [[htbp]
 \begin{center}
\includegraphics[width=1\textwidth,angle=0]{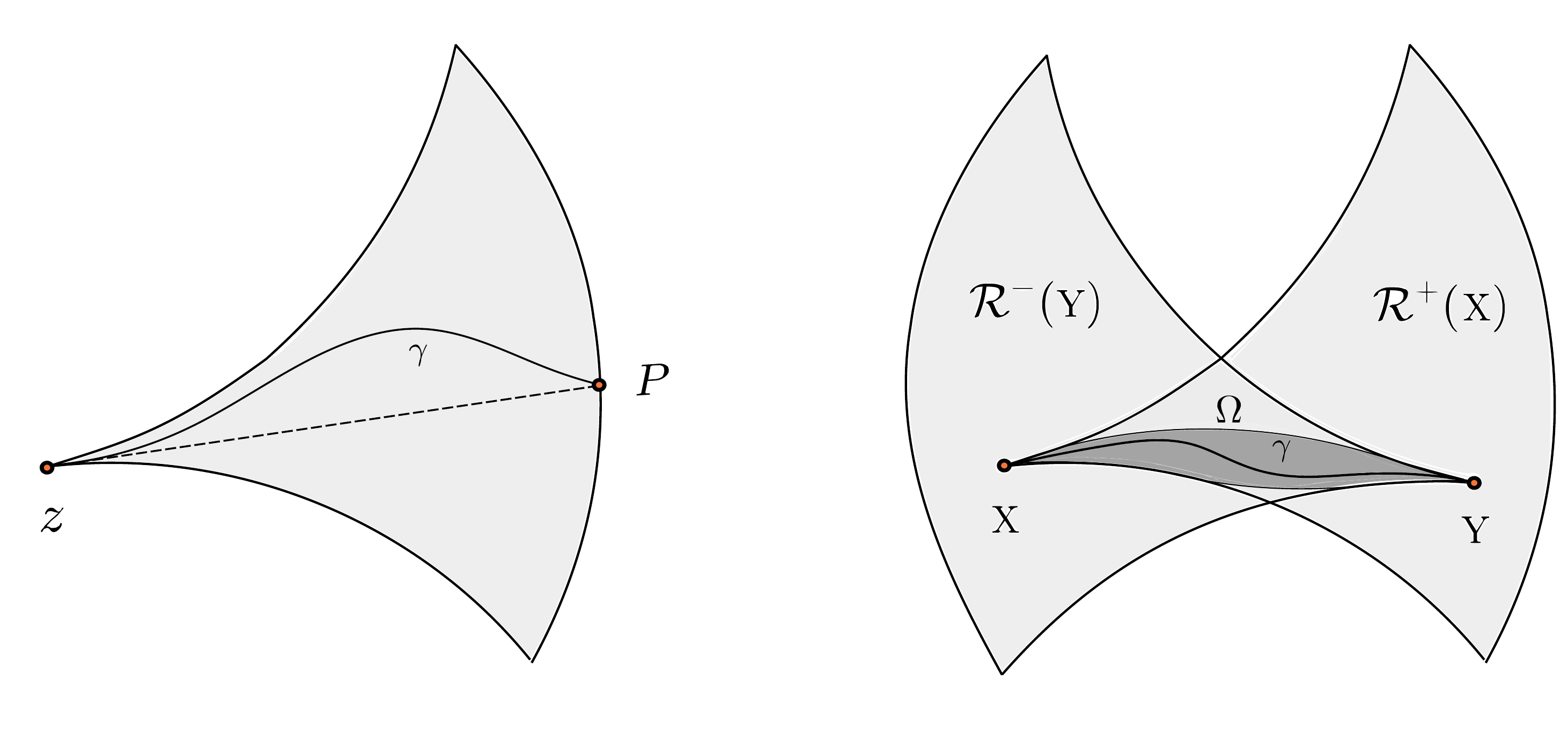}
\end{center}
\caption{Left: Applying Lemma \ref{seg} to $\gamma$ in ${\mathcal R}^+(\mbox{\sc z})$. Right: From \cite{paperc} we conclude that a fragment and replacement are trapped in a compact subset $\Omega \subset{\mathcal R}^+(\mbox{\sc x})\cap{\mathcal R}^-(\mbox{\sc y})$. }
\label{figregform}
\end{figure}}

The definition of a $cs$ path is such that the points where the curvature is not defined are precisely those between concatenated {\sc c} and {\sc s} components. In a {\sc csc} path call these points $P$ and $Q$ (see Figure \ref{figfunlem}).

\begin{proposition} \label{construct} For a fragment in $\Gamma(\mbox{\sc x,y})$ there exists a {\sc csc} path in $\Gamma(\mbox{\sc x,y})$ having circular components of length less than $\pi$.
 \end{proposition}

\begin{proof} Let $\gamma $ be a fragment with endpoint condition $\mbox{\sc x,y}\in T{\mathbb R}^2$. By Corollary \ref{nonedir} fragments do not have a negative direction, therefore, there is a {\it unique} line segment from $x$ tangent to $\mbox{\sc C}_l(\mbox{\sc y})$ or $\mbox{\sc C}_r(\mbox{\sc y})$. Consider the segment $\overline{xy}$. We can have three scenarios: $\overline{xy}$ crosses $\mbox{\sc C}_l(\mbox{\sc x})$, $\overline{xy}$ crosses $\mbox{\sc C}_r(\mbox{\sc x})$ or $\overline{xy}$ is tangent to $\mbox{\sc C}_l(\mbox{\sc x})$ and $\mbox{\sc C}_r(\mbox{\sc x})$. Symmetrically we can have: $\overline{yx}$ crosses $\mbox{\sc C}_l(\mbox{\sc y})$, $\overline{yx}$ crosses $\mbox{\sc C}_r(\mbox{\sc y})$ or $\overline{yx}$ is tangent to $\mbox{\sc C}_l(\mbox{\sc y})$ and $\mbox{\sc C}_r(\mbox{\sc y})$. First, suppose the segment $\overline {xy}$ crosses both $\mbox{\sc C}_r(\mbox{\sc x})$ and $\mbox{\sc C}_r(\mbox{\sc y})$. Observe that there is a common tangent line to $\mbox{\sc C}_r(\mbox{\sc x})$ and $\mbox{\sc C}_r(\mbox{\sc y})$ with respective tangent points say $P$ and $Q$. Also, there are arcs in $\mbox{\sc C}_r(\mbox{\sc x})$, $\mbox{\sc C}_r(\mbox{\sc y})$ from $x$ to $P$ and $Q$ to $y$ traversed counterclockwise. We define $\beta:I\rightarrow {\mathbb R}^2$ to be the arc-length parameterized {\sc csc} path of type {\sc rsr} if the segment $\overline {xy}$ crosses both $\mbox{\sc C}_r(\mbox{\sc x})$ and $\mbox{\sc C}_r(\mbox{\sc y})$ by considering the common tangent segment $\overline{PQ}$ to $\mbox{\sc C}_r(\mbox{\sc x})$ and $\mbox{\sc C}_r(\mbox{\sc y})$ and the arc from $x$ to $P$ and the arc from $Q$ to $y$ (see Figure \ref{figfunlem} right). In this fashion, $\beta$ corresponds to a bounded curvature path that starts from $x$,  travels along $\mbox{\sc C}_r(\mbox{\sc x})$ until reaching $P$, then travels along $\overline{PQ}$ and then travels along $\mbox{\sc C}_r(\mbox{\sc y})$ from $Q$ until reaching $y$. Similar reasoning applies for the construction of  {\sc lsr}, {\sc lsl}, {\sc rsl} path types under the other possible intersections of $\overline{xy}$ with the adjacent circles. By Corollary \ref{nonedir} fragments do not have a negative direction and by Lemma \ref{rad} the replacement has circular arcs of length less than $\pi$. 
\end{proof}

The following key result is proved by a direct projection argument. The idea is to divide up a fragment into at most three pieces and then project such pieces onto the replacement constructed in Proposition \ref{construct}. We conclude that the replacement is never longer than the fragment.

\begin{lemma}\label{fundlemma} The length of a replacement is at most the length of the associated fragment with equality if and only if these paths are identical (see Figure \ref{figfunlem}).
\end{lemma}

\begin{proof} Consider a fragment $\gamma $ with endpoint condition $\mbox{\sc x,y}\in T{\mathbb R}^2$. Without loss of generality consider $x=(0,0)$ and $X=(1,0)$. Construct a replacement path $\beta$ as in Proposition \ref{construct} and suppose that $\gamma \neq \beta$. Denote by $L_1$ the line passing through the center of the first adjacent arc in $\beta$ and the common point $P$ between the first and the second component of $\beta$ and denote by $L_2$ the line passing through the center of the second arc of $\beta$ and the common point $Q$ between the second and the third component of $\beta$. Observe that $L_1$ and $L_2$ are parallel lines.  By virtue of  Proposition \ref{r1r3pr} we have that $\gamma$ must lie in ${\mathcal R}^+(\mbox{\sc x})$ and in particular $\gamma$ lies outside the interior of the adjacent circles associated with {\sc x} and {\sc y}. By continuity, the path $\gamma$ must cross the lines $L_1$ and $L_2$ at, say, the points $O$ and $N$ respectively (see Figure \ref{figfunlem}). Denote by $\gamma_1$ the portion of $\gamma$ in between $x$ and the first time $\gamma$ intersects $L_1$; denote by $\gamma_2$ the portion of $\gamma$ in between the first time $\gamma$ intersects $L_1$ and the first time $\gamma$ intersects $L_2$ and denote by $\gamma_3$ the portion of $\gamma$ in between the first time $\gamma$ intersects $L_2$ and $y$. Suppose $\gamma \neq \beta$ at some point in $\gamma_1$. By Proposition \ref{r1r3pr} the fragment $\gamma$ does not intersect $\partial{\mathcal R}^+(\mbox{\sc x})$ and therefore $P\neq O$ (also $N\neq Q$). By virtue of Lemma \ref{rad} the length of $\gamma_1$ is greater than the length of $\beta$ (the angle traveled in the first component of $\beta$ from $x$ to $P$). By applying Lemma \ref{seg} to $\gamma_2$ we conclude that the length of $\gamma$ in between $O$ and $N$ is greater than the length of the segment $PQ$. By applying Lemma \ref{rad} (as we did for $\gamma_1$) we have that the length of $\gamma_3$ is greater than or equal to the length of $\beta$ in between $Q$ and $y$. Therefore, $\gamma$ must be longer than $\beta$. The cases $\gamma \neq \beta$ at some point in $\gamma_2$ or $\gamma_3$ are proven identically as above. \end{proof}

 Notice that fragments have length less than $1$. Such a bound is sufficient to guarantee that the replacement is a {\sc csc} path possibly with some components of zero length. On the other hand, {\sc csc} paths can be sometimes constructed for longer pieces in a bounded curvature path.

{ \begin{figure} [[htbp]
 \begin{center}
\includegraphics[width=1\textwidth,angle=0]{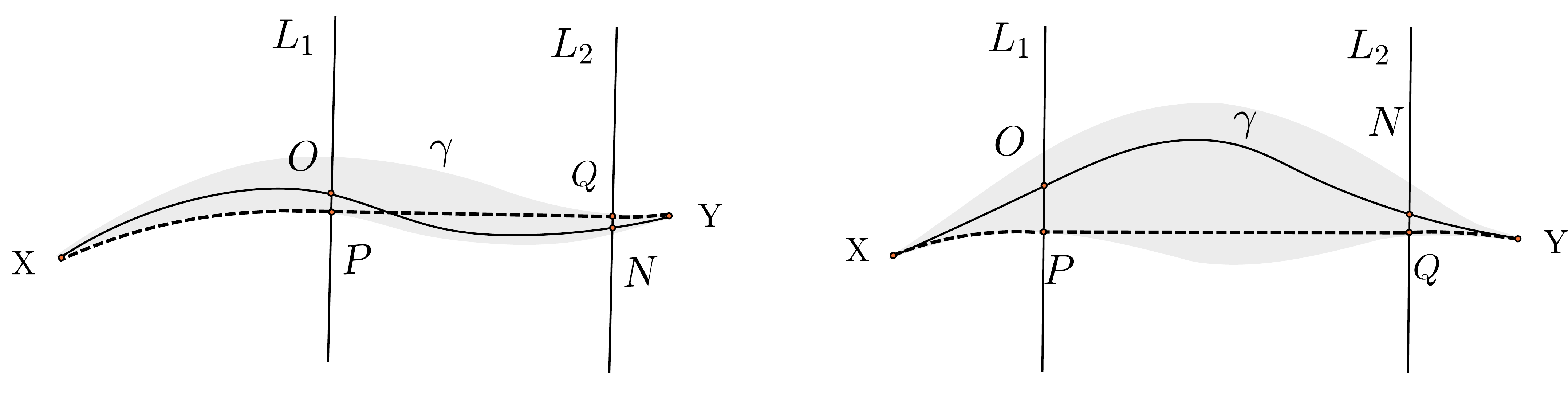}
\end{center}
\caption{The points $P$ and $Q$ are the points in the dashed {\sc csc} path where the curvature is not defined. The dashed trace corresponds to the replacement $\beta$ constructed in Lemma \ref{fundlemma}. The grey regions correspond to $\Omega$.}
 \label{figfunlem}
\end{figure}}

\begin{proposition} \label{csform} A path $\gamma \in \Gamma(\mbox{\sc x,y})$ may be replaced by a $cs$ path of length at most the length of $\gamma$.
\end{proposition}

\begin{proof} Consider a fragmentation for $\gamma \in \Gamma(\mbox{\sc x,y})$. By applying Lemma \ref{fundlemma} to each fragment we obtain a bounded curvature path of length at most the length of $\gamma$ being a finite number of concatenated {\sc csc} paths.  
\end{proof}

\section{Length Minimising Bounded Curvature Paths} \label{minimalembedded}

Here we continue the process of reducing bounded curvature paths in terms of length and complexity by looking at larger pieces in a $cs$ paths. We conclude in particular that length minimising bounded curvature paths have complexity at most three.

\begin{definition} A component of type ${\mathscr C}_1$ is a {\sc cscsc} path as shown in Figure \ref{figrep1} left and center. A component of type ${\mathscr C}_2$ is a {\sc csccsc} path as shown in Figure \ref{figrep1} right. \end{definition}

{ \begin{figure} [[htbp]
 \begin{center}
\includegraphics[width=1\textwidth,angle=0]{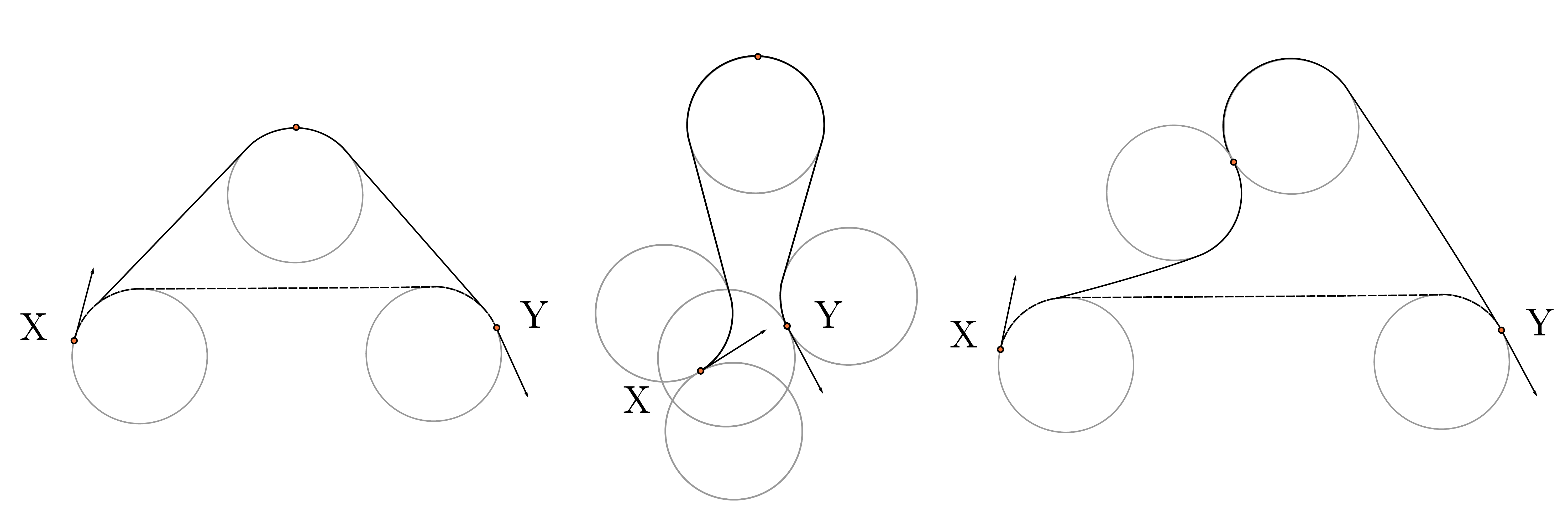}
\end{center}
\caption{Examples of components of type ${\mathscr C}_1$ and ${\mathscr C}_2$. The dashed traces in the left and right figures are replacements. The middle illustration corresponds to a non-admissible component of type ${\mathscr C}_1$. Note that no replacement can be constructed for such a component.}
 \label{figrep1}
\end{figure}}

\begin{definition}\label{norepcomp} Let $\gamma \in \Gamma(\mbox{\sc x,y})$ be a component of type ${\mathscr C}_1$ (or ${\mathscr C}_2$). Then $\gamma$ is called {\it admissible} if a replacement can be constructed in $\Gamma(\mbox{\sc x,y})$. A component of type ${\mathscr C}_1$ (or ${\mathscr C}_2$) that is not admissible is called a non-admissible component (see Figure \ref{figrep1}). \end{definition}

\begin{proposition}\label{lengthred} Given $\mbox{\sc x,y} \in T{\mathbb R}^2$. A $\it cs$ path in $\Gamma(\mbox{\sc x,y})$ containing an admissible component as a sub path can be replaced by another $\it cs$ path in $\Gamma(\mbox{\sc x,y})$ with less complexity and with the length of the latter being at most the length of the former.
\end{proposition}

\begin{proof} Consider a $cs$ path containing an admissible component of type ${\mathscr C}_1$ (or ${\mathscr C}_2$). By applying the same construction in Lemma \ref{fundlemma} to the $cs$ path in between the component of type ${\mathscr C}_1$ (or ${\mathscr C}_2$)  (see Figure \ref{figrepscsnomin} left) the length of the replacement is seen to be at most the length of the component of type ${\mathscr C}_1$ (or ${\mathscr C}_2$) concluding the proof. 
\end{proof}

{ \begin{figure} [[htbp]
 \begin{center}
\includegraphics[width=1\textwidth,angle=0]{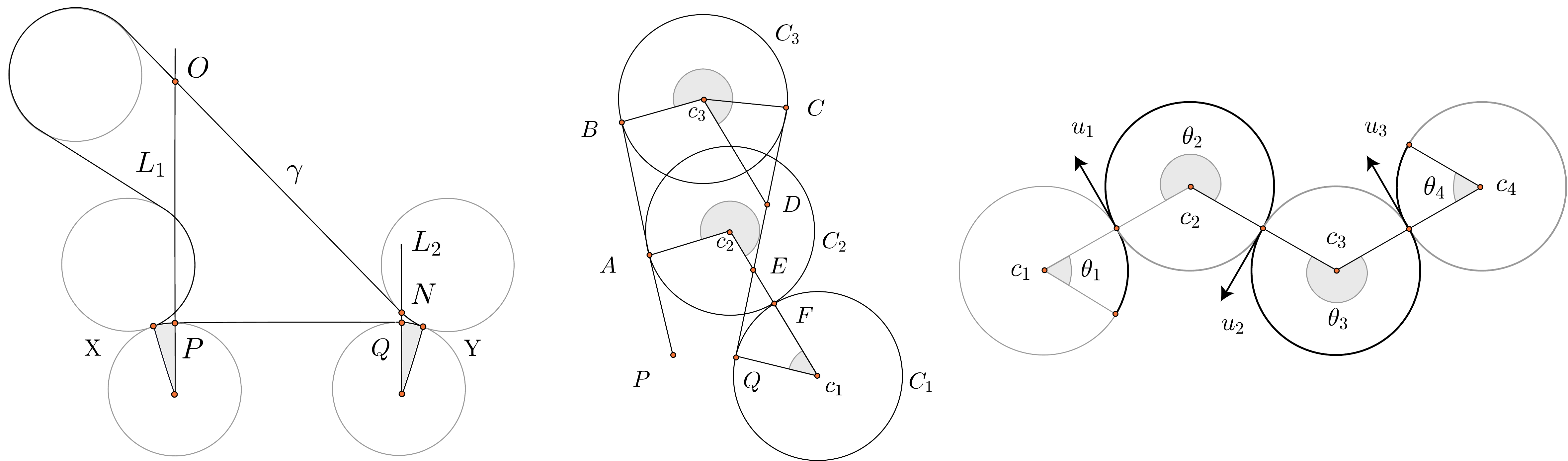}
\end{center}
\caption{Left: The restriction of a $cs$ path $\gamma$ to an admissible component of type ${\mathscr C}_1$. Lemma \ref{fundlemma} is trivially adapted to a component of type ${\mathscr C}_1$ (or ${\mathscr C}_2$) (see notation in Figure \ref{figfunlem}). Center: The notation in Proposition \ref{scsccsnomin}. The path $\gamma$ is the {\sc scs} path from $P$ to $Q$. Right: The notation for the non optimality of {\sc cccc} paths.}
 \label{figrepscsnomin}
\end{figure}}

\begin{theorem} \label{scsccsnomin} Components of type ${\mathscr C}_1$ and ${\mathscr C}_2$ are not length minimisers.
\end{theorem}

\begin{proof} Consider a component of type $\mathscr{C}_1$. If the component is admissible then by applying Proposition \ref{lengthred} the result follows. If the component of type $\mathscr{C}_1$ is non-admissible then consider the {\sc scs} part of the component of type $\mathscr{C}_1$, denote it by $\gamma$, and without loss of generality suppose the line segments in $\gamma$ have the same length. Referring to Figure \ref{figrepscsnomin} (center) for notation, we proceed to construct a path shorter than $\gamma$. Note that $\gamma$ is the {\sc scs} path from $P$ to $Q$ with the length of $\overline{PB}$ and $\overline{CQ}$ being the same. Denote by $C_1$ the left adjacent circle at $Q$. Let $C_2$ be the circle that is simultaneously tangent to $C_1$ and $\overline{PB}$. Let $C_3$ be the circle containing the middle component of $\gamma$; denote the centers of these circles in lowercase. The parallel line to $\overline{c_1\,c_2}$ passing through $c_3$ intersects $\overline{CQ}$ at $D$. The segment $\overline{c_1\,c_2}$ intersects $\overline{CQ}$ at $E$. Since $C_3$ is above $C_2$ (both circles being tangent to $\overline{PB}$) then by applying Lemma \ref{rad} we conclude that the length of $\gamma$ between the points $B$ and $D$ is greater than the length of the arc in $C_2$ between $A$ and $F$. Again by Lemma \ref{rad} we conclude that the length of the shorter arc between ${Q}$ and $F$ in $C_1$ is less than the length of the segment $\overline{QE}$. Denote by $\delta$ the {\sc scc} path from $P$ to $Q$ with first component $\overline{PA}$, second component the arc $AF$ in $C_2$ and third component the arc $FQ$ lying in $C_1$. We conclude that $\mathcal{L}(\delta)<\mathcal{L}(\gamma)$ implying that the component of type $\mathscr{C}_1$ is not a length minimiser. If a component of type $\mathscr{C}_2$ is admissible then by applying Lemma \ref{rad} the result follows. If a component of type $\mathscr{C}_2$ is non-admissible, then by applying an identical argument as in the previous paragraphs the result follows.
\end{proof}

\begin{corollary}\label{coroscs} The {\sc scs} and {\sc ccs} (or {\sc scc}) paths are not length minimisers.
\end{corollary}
\begin{proof} By the construction in Theorem \ref{scsccsnomin} we have that {\sc scs} paths are not length minimisers. The proof that {\sc ccs} (or {\sc scc}) paths are also not length minimisers follows from an identical construction as the one used for the non-admissible case in Theorem \ref{scsccsnomin}. We leave the details to the reader.
\end{proof}

\begin{proposition} \label{cccmnopt} A {\sc ccc} path with middle arc of length less than or equal to $\pi$ is not a length minimiser.
\end{proposition}

\begin{proof}  Without loss of generality consider an {\sc lrl} path. An {\sc lrl} path with middle component of length less than or equal to $\pi$ has the center of the circle containing {\sc r} below the line joining $c_l(\mbox{\sc x})$ and $c_l(\mbox{\sc y})$. Consider the adjacent circles $\mbox{\sc C}_l(\mbox{\sc x})$ and $\mbox{\sc C}_l(\mbox{\sc y})$ and construct a replacement as in Proposition \ref {construct}. Apply Lemma \ref{fundlemma} to the {\sc rlr} path to conclude the statement. If we consider an {\sc rlr} path with middle component of length less than or equal to $\pi$ then consider $\mbox{\sc C}_r(\mbox{\sc x})$ and $\mbox{\sc C}_r(\mbox{\sc y})$ and proceed as before.
\end{proof}

\begin{proposition}\label{ccccnopt}  The $cs$ paths of complexity four are not length minimisers.
\end{proposition}
\begin{proof} If a complexity four path contains exactly one line segment then it must contain a {\sc ccs} or {\sc scc} component. By virtue of Corollary \ref{coroscs} we have that such a component is not a length minimiser. If a complexity four path contains exactly two line segments then it must contain an {\sc scs} component. Again by Corollary \ref{coroscs} we have that such a component is not a length minimiser. 

For a proof that {\sc cccc} paths are not length minimisers, we use a variational argument. In particular, we show that if a {\sc cccc} path is a critical point of the length function, then it is unstable. So that it is evident how the equations scale as the minimum turning radius $\kappa$ varies, we do not set $\kappa$ equal to unity. Let $\gamma$ be a {\sc cccc} path whose arcs have lengths $\theta_1(t), \theta_2(t), \theta_3(t)$ and $\theta_4(t)$ and centres $c_1(t), c_2(t), c_3(t)$ and $c_4(t)$, respectively. For succinctness, we suppress the dependency on $t$. Hence the length of $\gamma$ is given by
\begin{equation*}
\mathcal{L}(\gamma) = \theta_1 + \theta_2 + \theta_3 + \theta_4.
\end{equation*}

Consider a perturbation of $\gamma$ to a nearby path of type {\sc cccc}. The first variation of length is given by
\begin{equation*}
\dot{\mathcal{L}}(\gamma) = \langle(c_2 - c_1)',\hat{u}_1\rangle + \langle(c_3 - c_2)',\hat{u}_2\rangle + \langle(c_4 - c_3)',\hat{u}_3\rangle,
\end{equation*}
where the unit vectors $\hat{u}_i$ are tangent to the circular arcs of $\gamma$ with centres $c_i$ and $c_{i+1}$ and have the same orientation as $\gamma$ (see Figure \ref{figrepscsnomin} right). Note that the lengths $\|c_{i+1} - c_i\|$ remain constant throughout the perturbation. Hence, the vector $(c_{i+1} - c_i)'$ must be perpendicular to the vector $c_{i+1} - c_i$. In other words, the vector $(c_{i+1} - c_i)'$ must be parallel to the vector $\hat{u}_i$. So we obtain $\langle(c_{i+1} - c_i)',\hat{u}_i\rangle = \pm\|(c_{i+1} - c_i)'\|$, where the sign of the inner product depends on the direction of variation. In particular, the inner products $\langle(c_2 - c_1)',\hat{u}_1\rangle$ and $\langle(c_4 - c_3)',\hat{u}_3\rangle$ are of opposite sign. On the other hand, the inner product $\langle(c_3 - c_2)',\hat{u}_2\rangle$ may change sign relative to $\langle(c_2 - c_1)',\hat{u}_1\rangle$ and $\langle(c_4 - c_3)',\hat{u}_3\rangle$. Note that the centres $c_1$ and $c_4$ are fixed. Hence both $c_1' = 0$ and $c_4' = 0$. At the critical point, we set the derivative of the length function $\mathcal{L}(\gamma)$ equal to zero and obtain one of the following two equations:
\begin{equation*}
\dot{\mathcal{L}}(\gamma) = \|c_2'\| + \|c_3' - c_2'\| - \|c_3'\| = 0,
\end{equation*}
or
\begin{equation*}
\dot{\mathcal{L}}(\gamma) = \|c_2'\| - \|c_3' - c_2'\| - \|c_3'\| = 0.
\end{equation*}

In either case, the vectors $c_2'$ and $c_3'$ must be parallel by the triangle inequality. Since the vector $c_3 - c_2$ cannot undergo any change in length, the perturbation vectors $c_2'$ and $c_3'$ must in fact be equal. Note that the perturbation vectors $c_2'$ and $c_3'$ are tangent to the circles of radius $2\kappa$ with centres $c_1$ and $c_4$. It is geometrically obvious that a stationary configuration is obtained only in the symmetric situation where $\hat{u}_1 = \hat{u}_3$. In order to demonstrate that this configuration corresponds to an unstable critical point of the length function, we will need to study the second variation of length. By the Leibniz rule,
\begin{equation*}
\ddot{\mathcal{L}}(\gamma) = \langle(c_2 - c_1)'',\hat{u}_1\rangle + \langle(c_2 - c_1)',\hat{u}_1'\rangle + \langle(c_3 - c_2)'',\hat{u}_2\rangle + \langle(c_3 - c_2)',\hat{u}_2'\rangle + \langle(c_4 - c_3)'',\hat{u}_3\rangle + \langle(c_4 - c_3)',\hat{u}_3'\rangle.
\end{equation*}

By the same reasoning as above, the vector $(c_{i+1} - c_i)'$ must be perpendicular to the vector $\hat{u}_i'$, and  both $c_1'' = 0$ and $c_4'' = 0$. Moreover, at the critical point, $\hat{u}_1 = \hat{u}_3$. Therefore,
\begin{equation*}
\ddot{\mathcal{L}}(\gamma) = \langle(c_3 - c_2)'',\hat{u}_2\rangle - \langle(c_3 - c_2)'',\hat{u}_1\rangle.
\end{equation*}

The vector $(c_3 - c_2)''$ makes an obtuse angle with the vector $\hat{u}_2$ throughout the perturbation and is parallel to the vector $\hat{u}_2$ at the critical point. To see this, note that the component of $(c_3 - c_2)''$ parallel to $c_3 - c_2$ is equal to $-|\langle(c_3 - c_2)',\hat{u}_2\rangle|^2/2\kappa$, and akin to the centripetal acceleration of the vector $c_3 - c_2$. At the critical point, $c_2' = c_3'$. Hence, the vector $c_3 - c_2$ is not subject to any rotation and the component of $(c_3 - c_2)''$ perpendicular to $c_3 - c_2$ must be zero. Therefore, 
\begin{equation*}
\ddot{\mathcal{L}}(\gamma) = -\|(c_3 - c_2)''\| - \langle(c_3 - c_2)'',\hat{u}_1\rangle < 0,
\end{equation*}
and the configuration is unstable.
\end{proof}

\begin{corollary} \label{nomin}  $cs$ paths of complexity greater than $3$ are not length minimisers.
\end{corollary}
\begin{proof} Immediate from Proposition \ref{ccccnopt}.
\end{proof}

Now we proceed to establish the classification of length minimisers in spaces of planar bounded curvature first obtained in \cite{dubins 1}.

\begin{theorem} \label{embdub}Choose $\mbox{\sc x,y} \in T{\mathbb R}^2$. A length minimising bounded curvature path in $\Gamma(\mbox{\sc x,y})$ is either a {\sc ccc} path having its middle component of length greater than $\pi $ or a {\sc csc} path. Some of the circular arcs or line segments can have zero length.
\end{theorem}

\begin{proof} Choose $\mbox{\sc x,y} \in T{\mathbb R}^2$ and consider a fragmentation for  $\gamma \in\Gamma(\mbox{\sc x,y})$. If $\gamma$ is not a $cs$ path then by applying Lemma \ref{fundlemma} to each fragment we obtain a $cs$ path in $\Gamma(\mbox{\sc x,y})$ shorter than $\gamma$. We conclude that $\gamma$ is not a length minimiser in $\Gamma(\mbox{\sc x,y})$. If $\gamma$ is a $cs$ path of complexity greater than or equal to $4$, by Corollary \ref{nomin} we conclude that $\gamma$ is not a length minimiser in $\Gamma(\mbox{\sc x,y})$. If the complexity of $\gamma$ is exactly $3$ then by Corollary \ref{coroscs} we have that {\sc scs} and {\sc ccs} (or {\sc scc}) paths are not length minimisers. By applying Proposition \ref{cccmnopt} we conclude the proof.
\end{proof}

The paths in the previous result are called {\it Dubins paths} in honour of Lester Dubins who proved Theorem \ref{embdub} for the first time in 1957 in \cite{dubins 1}.

\section{Remarks on generalisations}
By applying the Pontryagin maximum principle to a time optimal control system F. Monroy-P\'{e}rez in \cite{mon} characterised the length minimisers in 2-dimensional homogeneous spaces of constant curvature. D. Mittenhuber also using an argument with control theoretical flavour obtained Dubins result in the hyperbolic 2-space. 

Notice all the definitions in this work can be adapted for other surfaces. In particular, without much effort Lemma \ref{seg} and Lemma \ref{rad} can be adapted to paths in 2-dimensional homogeneous spaces of constant curvature.

\subsection{Remarks on the hyperbolic case.} 

Here Lemma \ref{fundlemma}  is a little bit more complicated to prove. But otherwise exactly the same approach works.
The paths of constant curvature in the hyperbolic plane are semicircles orthogonal to the real axis including the orthogonal upper half lines. Notice in this case we cannot rescale to only deal with $\kappa = 1$ since rescaling changes the
underlying curvature. So we need to always work with arcs of circles of appropriate fixed radius depending on the choice of $\kappa$.

\subsection{Remarks on the elliptic case.}

For the positive curvature case we propose to work on a 2-sphere. Clearly there is a scaling issue to be addressed. This is due to the ratio between the curvature bound and the radius of the sphere.  An identical method as the one employed for the euclidean or hyperbolic case (see Lemma \ref{fundlemma}) does not work here. Due to the geometry of the sphere, the length of a fragment in between $L_1$ and $L_2$ as in Figure \ref{figfunlem} (but for the spherical case) could eventually be shorter than the projected path in between $L_1$ and $L_2$. This situation can be easily overcome by considering a sequence of replacements.

\bibliographystyle{amsplain}
   
\end{document}